\documentclass[11pt, reqno]{amsart}

\usepackage[english]{babel}

\usepackage{amsmath, amsthm, amssymb}
\usepackage{graphicx}
\usepackage[hidelinks]{hyperref}
\usepackage{txfonts}
\usepackage{stmaryrd}
\usepackage{enumitem}
\usepackage{mathrsfs}
\usepackage{bbm}
\usepackage{tikz-cd}
\numberwithin{equation}{section}

\newtheorem{theorem}{Theorem}[section]
\newtheorem{proposition}[theorem]{Proposition}
\newtheorem{lemma}[theorem]{Lemma}

\theoremstyle{definition}

\newtheorem{definition}[theorem]{Definition}
\newtheoremstyle{customNumber}
     {}          
     {}          
     {\itshape}  
     {}          
     {\bfseries} 
     {.}         
     { }         
     {\thmname{#1}\thmnumber{ #2}\thmnote{ #3}}
\theoremstyle{customNumber}

\renewcommand{\rho}{\varrho}

\newcommand{\norm}[1]{\lVert#1\rVert}
\newcommand{\abs}[1]{\lvert#1\rvert}

\newcommand{\on}{\:\mbox{\rule{0.1ex}{1.2ex}\rule{1.1ex}{0.1ex}}\:}
\newcommand{\bb}[1]{\llbracket #1\rrbracket}

\DeclareMathOperator{\N}{\mathbb{N}}
\DeclareMathOperator{\R}{\mathbb{R}}

\DeclareMathOperator{\Lip}{Lip}

\DeclareMathOperator{\sgn}{sgn}

\DeclareMathOperator{\spt}{spt}

\DeclareMathOperator{\mass}{\mathbf{M}}

\DeclareMathOperator{\bL}{\mathbf{L}}
\DeclareMathOperator{\id}{id}

\DeclareMathOperator{\vol}{Vol}

\DeclareMathOperator{\Haus}{\mathscr{H}}
\DeclareMathOperator{\Leb}{\mathscr{L}}

\DeclareMathOperator{\Jac}{Jac}

\DeclareMathOperator{\loc}{loc}
\DeclareMathOperator{\md}{md}

\DeclareMathOperator{\weakd}{wd}
\DeclareMathOperator{\J}{J}
\DeclareMathOperator{\ir}{ir}
\DeclareMathOperator{\crr}{cr}

\DeclareMathOperator{\htt}{ht}
\DeclareMathOperator{\bus}{b}

\newdimen\vintkern\vintkern11pt
\def\vint{-\kern-\vintkern\int}

\renewcommand{\epsilon}{\varepsilon}
\def\XXint#1#2#3{{\setbox0=\hbox{$#1{#2#3}{\int}$ }
\vcenter{\hbox{$#2#3$ }}\kern-.6\wd0}}

\usepackage{etoolbox}

\makeatletter
\patchcmd{\@setaddresses}{\indent}{\noindent}{}{}
\patchcmd{\@setaddresses}{\indent}{\noindent}{}{}
\patchcmd{\@setaddresses}{\indent}{\noindent}{}{}
\patchcmd{\@setaddresses}{\indent}{\noindent}{}{}
\makeatother


\thanks{The author gratefully acknowledges financial support by the MPIM Bonn.}

\author{Giuliano Basso}
\address{
Max Planck Institute for Mathematics,
Vivatsgasse 7,
53111 Bonn,
Germany}
\email{basso@mpim-bonn.mpg.de}

\title{Finsler currents}

\begin{document}
\maketitle

\begin{abstract}
We propose a slight variant of Ambrosio and Kirchheim's definition of a metric current. We show that with this new definition it is possible to obtain certain volume functionals from Finsler geometry as mass measures of currents. As an application, we obtain a whole family of extendibly convex \(n\)-volume densities. This family includes the circumscribed Riemannian and the mass* volume densities.
\end{abstract}

\section{Introduction}
\subsection{Motivation}
In this article, we propose a slight variant of the definition of a metric current (see Definition~\ref{def:main-def} below). Building on ideas of de Giorgi, metric currents were introduced in \cite{ambrosio-2000} by Ambrosio and Kirchheim to generalize the classical theory of currents in Euclidean space to arbitrary complete metric spaces. The theory has been employed with success in various settings (see, for example, \cite{goldhirsch2021characterizations, huang-2022, sormani-2017, kleiner-2020, song2023entropy, sormani--2011, wenger-2005, wenger-2008}). Roughly speaking, a current \(T\) on a metric space \(X\) is a multilinear functional \(T\colon \Lip_b(X, \R)\times [\Lip(X, \R)]^{\,n}\to \R\) satisfying certain minimal assumptions such as continuity and a finite mass axiom. Thanks to the latter, to any current \(T\) there is associated a finite Borel measure \(\norm{T}\) on \(X\), called the mass measure of \(T\). As a particularly important example, we note that any closed oriented connected Riemannian \(n\)-manifold \(M\) induces an integral current \(\bb{M}\) in a natural way via integration of Lipschitz differential \(n\)-forms. The mass measure of \(\bb{M}\) is equal to the Riemannian volume measure of \(M\). 

Unlike for Riemannian manifolds, for Finsler manifolds there are several non-equivalent notions of volume. Most notably the volume functionals of Busemann and Holmes--Thompson, Gromov's mass* volume, and Ivanov's inscribed Riemannian volume. See \cite{thompson-2004, ivanov-2009} for basic properties of these volumes. As is well-known, the usual definition of a current recovers Gromov's mass* volume functional for Finsler manifolds. Indeed, if \(M\) is as above and \(\Phi \colon TM \to \R\) is a Finsler structure on \(M\) with associated Finsler metric \(d_\Phi\), then \(T=\bb{M}\) is also a current in \((M, d_\Phi)\) and \(\norm{T}
\) is equal to the mass* volume of \((M, \Phi)\).  

Using our modified definition of currents it is possible to recover other volume functionals from Finsler geometry. In particular, we show that it is possible to recover the circumscribed Riemannian volume functional as the mass measure of currents; see Proposition~\ref{prop:explicit-description-L-2}. This opens the door to study filling problems in Finsler geometry by using results and techniques from the theory of metric currents. 

\subsection{Definition of \(\bL\)-currents}
Let \(X\) be a complete metric space and \(\mathcal{D}^n(X)\) the set of all pairs \((f, \pi)\) with \(f\colon X \to \R\) a bounded Lipschitz function and \(\pi \colon X \to \R^n\) a Lipschitz map. If necessary, we will sometimes tacitly identify \(\mathcal{D}^n(X)\) with \(\Lip_{b}(X)\times \prod_{i=1}^n \Lip(X)\). Moreover, we write \(\pi^{(i)}\colon X \to \R\) for the \(i\)-th coordinate function of \(\pi\). We propose the following variant of Ambrosio and Kirchheim's definition \cite{ambrosio-2000} of a metric current.

\begin{definition}\label{def:main-def}
Let \(\mathbf{L}\colon \Lip(X, \R^n)\to \R\) be a semi-norm. A multilinear functional \(T\colon \mathcal{D}^n(X)\to \R\) is called \(\mathbf{L}\)-current if the following holds:
\vspace{0.25em}
\begin{enumerate}[label=\arabic*.)]
\item\label{it:one} If \((\pi_i)\) converges pointwise to \(\pi\in \Lip(X, \R^n)\) and \(\sup_i \mathbf{L}(\pi_i)\) is finite, then \(T(f, \pi_i)\to T(f, \pi)\) as \(i\to \infty\). \vspace{0.25em}
\item\label{it:two} \(T(f, \pi)=0\) if for some \(i\in \{1, \ldots, n\}\) the function \(\pi^{(i)}\) is constant on an open neighborhood of \(\spt(f)\). \vspace{0.25em}
\item\label{it:three} There exists a finite Borel measure \(\mu\) on \(X\) such that
\begin{equation}\label{eq:mass-estimate}
\abs{T(f, \pi)} \leq \mathbf{L}(\pi)^{\,n} \, \int_X \abs{f}\, d\mu
\end{equation}
for all \((f, \pi)\in \mathcal{D}^n(X)\).
\end{enumerate}
\end{definition}

The minimal measure \(\mu\) satisfying \eqref{eq:mass-estimate} is called mass measure of \(T\) and is denoted by \(\norm{T}_{\mathbf{L}}\). Any function \(\mathbf{L}\) as in Definition~\ref{def:main-def} is called \textit{Lipschitz norm}. We remark that analogous functions defined on \(C(X, \R)\) are of central importance in the theory of quantum metric spaces (see \cite{rieffel-1999, rieffel-2004} for more information).

Given a map \(f\colon X \to Y\) between metric spaces, we let \(\Lip(f)\) denote the best Lipschitz constant of \(f\). Recall that \(T\) is a \textit{current} if it satisfies Definition~\ref{def:main-def} with \(\sup_i \mathbf{L}(\pi_i)\) and \(\mathbf{L}(\pi)^n\) replaced by \(\sup_{i, j} \Lip(\pi_i^{(j)})\) and \(\prod_i \Lip(\pi^{(i)})\), respectively.  As a first example, we consider the Lipschitz norm \(\bL^\infty\), defined by
 \[
 \mathbf{L}^{\infty}(\pi)= \max_{i=1, \ldots, n} \Lip(\pi^{(i)}).
 \]
The following result shows that the usual definition of a metric current can be replaced by Definition~\ref{def:main-def} with \(\bL=\bL^\infty\), which has the advantage of being more symmetric.

\begin{proposition}\label{prop:equivalent-to-ak}
Let \(X\) be a complete metric space and \(T\colon \mathcal{D}^n(X)\to\R\) a multilinear functional. Then \(T\) is a current if and only if it is an \(\mathbf{L}^\infty\)-current.  Moreover, \(\norm{T}=\norm{T}_{\mathbf{L}^\infty}\).
\end{proposition}

The proof is straightforward and can be found in Section~\ref{section:six}. A natural class of Lipschitz norms to consider are those induced by norms on \(\R^n\). Let \(\sigma\) be a norm on \(\R^n\). We set \(V=(\R^n, \sigma)\) and denote by \(I\colon \R^n \to V\) the identity map. Then
\begin{equation}\label{eq:main-example-Lip-norm}
\mathbf{L}^{\sigma}(\pi)= \Lip( I\circ \pi)
\end{equation}
is a Lipschitz norm on \(\Lip(X, \R^n)\). Notice that \(\mathbf{L}^\infty=\mathbf{L}^{\abs{\, \cdot \,}_{\infty}}\), where \(\abs{\, \cdot \,}_{\infty}\) denotes the supremum norm on \(\R^n\). The following result can be established by elementary arguments. 

\begin{proposition}\label{prop:main}
Let \(X\) be a complete metric space and \(T\colon \mathcal{D}^n(X)\to\R\) a multilinear functional. Furthermore, let \(\sigma\) be a norm on \(\R^n\) normalized so that the Lebesgue measure of its unit ball is equal to the volume of the Euclidean \(n\)-ball. Then \(T\) is a current if and only if it is an \(\mathbf{L}^\sigma\)-current. Moreover, 
\begin{equation}\label{eq:mass-inequality-comparison}
C^{-1}\cdot \norm{T}_{\bL^\sigma} \leq \,\norm{T} \leq C\cdot  \norm{T}_{\bL^\sigma},
\end{equation}\label{eq:desired-main}
where one can take \(C=n^n\).
\end{proposition}

In view of Proposition~\ref{prop:main}, it therefore seems appropriate to consider metric currents as functionals to which a whole family of volume measures can naturally be assigned. Depending on the situation, some of these measures will be more suitable than others. It seems to be an interesting question to investigate further under what conditions on \(\bL\), an  \(\bL\)-current is also a metric current.

 \subsection{Rectifiable \(\bL^\sigma\)-currents }

In this section we give an explicit description of \(\norm{T}_{\bL^\sigma}\) for rectifiable currents \(T\). Most major results in the theory of metric currents are formulated for such currents. In fact, in many cases only integer-rectifiable or even integral currents are considered. Most notably, Plateau's problem in metric spaces \cite{ambrosio-2000, wenger-2005, wenger-2014, schmidt-2013} is usually formulated for integral currents, a class of currents with particularly good compactness properties. Integer-rectifiable currents can be seen as generalizations of \(n\)-rectifable sets, which in turn should be thought of as a rough metric version of smooth \(n\)-manifolds.
Recall that a \(\Haus^n\)-measurable subset \(S\subset X\) of a metric space \(X\) is called \textit{\(n\)-rectifiable} if there are countably many Borel sets \(A_i\subset \R^n\) and Lipschitz maps \(\varphi_i\colon A_i \to X\) such that
\begin{equation}\label{eq:kirchheim-cover}
\Haus^n\Big( S \setminus \bigcup_{i=1}^\infty \varphi_i(A_i)\Big)=0.
\end{equation}


The definition of a rectifiable metric current from \cite{ambrosio-2000} can be adopted literally to our situation.

\begin{definition}
An \(n\)-dimensional \(\mathbf{L}\)-current \(T\) is called \textit{rectifiable} if its mass measure \(\norm{T}_{\mathbf{L}}\) vanishes on \(\Haus^n\)-negligible Borel sets and is concentrated on an \(n\)-rectifiable set.    
\end{definition}
 
Our main result gives an explicit description of \(\norm{T}_{\mathbf{L}^\sigma}\) for rectifiable currents, see Theorem~\ref{thm:main-2} below. This description relies on Jacobians, a notion which is in one-to-one correspondence with Finsler volume functionals. By definition, a Jacobian \(\J^\bullet\) is a real-valued map defined on the set \(\Sigma\) of all semi-norms on  \(\R^n\) satisfying three natural axioms; see Section~\ref{sec:four}. We refer to \cite{thompson-2004} for a general survey article on Jacobians and related notions. If \(\sigma\) is a fixed norm on \(\R^n\) with unit ball \(B_\sigma\),  then the following map defines a Jacobian. For every norm \(s\in \Sigma\) we set
\begin{equation}\label{eq:jacobi}
\J^{\sigma}(s)=\sup_{F} \frac{\Leb^n(B_\sigma)}{\mathscr{L}^n(F^{-1}(B_\sigma))},
\end{equation}
where the supremum is taken over all linear maps \(F\colon \R^n \to \R^n\) such that \(F^{-1}(B_\sigma)\) contains the unit ball of \(s\). For any other \(s\in \Sigma\) which is not a norm, we define \(\J^\sigma(s)=0\). If \(\sigma\) equals \(\abs{\, \cdot \,}_{\infty}\) or \(\abs{\,\cdot\,}_2\), this recovers the mass* Jacobian or the circumscribed Riemannian Jacobian, respectively.


In order to state our main result, we now recall how a Jacobian  naturally induces a volume measure on a rectifiable set. Kirchheim showed that the covering \eqref{eq:kirchheim-cover} of a rectifiable set by Lipschitz pieces can be improved (see \cite[Lemma~4]{kirchheim-1994}). If \(S\subset X\) is \(n\)-rectifiable, then there are compact subsets \(K_i\subset \R^n\) and bi-Lipschitz embeddings \(\varphi_i\colon K_i \to X\) with pairwise disjoint images such that \(S\) and \(\bigcup_{i=1}^\infty \varphi_i(K_i)\) are equal up to a set of \(\Haus^n\)-measure zero. Such collection \((K_i, \varphi_i)\) is called a bi-Lipschitz parametrization of \(S\). Suppose now that we are given such a parametrization. Then
\[
\mu_S^\bullet(A)=\sum_{i=1}^\infty \int_{\varphi_i^{-1}(A)} \J^\bullet(\md_p \varphi_i)\, \text{d}p
\]
defines a Borel measure on \(X\). Here, \(\md_p \varphi\) denotes the metric differential of \(\varphi\). We recall the necessary background from metric analysis in Section~\ref{sec:four}. The measure \(\mu_S^\bullet\) is called \textit{Finsler volume} associated to \(S\). It is not difficult to show that \(\mu_S^\bullet\) does not depend on the choice of bi-Lipschitz parametrization \((K_i, \varphi_i)\). 

Let \(\mu_S^\sigma\) denote the Finsler volume induced by the Jacobian \(J^\sigma\) defined in \eqref{eq:jacobi}.  As it turns out, for a rectifiable current \(T\) the mass measure \(\norm{T}_{\bL^\sigma}\) is absolutely continuous with respect to \(\mu^\sigma_S\).

\begin{theorem}\label{thm:main-2}
Let \(T\) be an \(n\)-dimensional, rectifiable \(\bL^\sigma\)-current in a complete metric space \(X\). Suppose \(\norm{T}_{\bL^\sigma}\) is concentrated on the \(n\)-rectifiable set \(S\subset X\). Then 
\[
 \norm{T}_{\bL^\sigma}= \abs{\theta} \, \mu^\sigma_S,
\]
where \(\theta\) denotes the multiplicity function of \(T\).
\end{theorem}
See Theorem~\ref{thm:formula-for-the-mass-measure} below for a more detailed version of the theorem. This generalizes the well-known fact that \(\norm{T}\) is absolutely continuous with respect to the Gromov mass\(\ast\) volume measure \(\mu^{m\ast}\); see \cite[Theorem 9.5]{ambrosio-2000}. In the next subsection, we collect some corollaries of Theorem~\ref{thm:main-2}. 

 

\subsection{Applications to Finsler volumes}\label{sec:applics}
We now proceed with two applications of \(\bL^\sigma\)-currents to the theory of volumes considered in Finsler geometry. If \(\sigma\) is equal to the standard Euclidean norm \(\abs{\,\cdot\,}_2\) on \(\R^n\), then \(J^\sigma\) recovers the circumscribed Riemannian Jacobian \(J^{\crr}\). To make the following results more accessible we formulate them only in terms of \(\J^{\crr}\). But it is important to keep in mind that they actually apply to all Jacobians \(J^\sigma\). In particular, Proposition~\ref{thm:extendabily-convex-allgemein} below provides a whole family of extendibly convex \(n\)-volume densities. 

We abbreviate \(\bL^2:=\bL^{\abs{\,\cdot\,}_2}\). Recall that any closed oriented connected Riemannian \(n\)-manifold \((M, g)\) naturally induces a current \(\bb{M}\) in \((M, d_g)\) by integrating Lipschitz \(n\)-forms. See \cite[Example~2.32]{sormani--2011} for the precise definition. A Finsler structure is a continuous map \(\Phi \colon TM \to \R\) such that its restriction to any tangent space is a norm. As with a Riemannian structure, a Finsler structure \(\Phi\) naturally induces a metric \(d_\Phi\) such that \((M, d_\Phi)\) is a complete metric space and the identity map \(\iota\colon (M, d_g) \to (M, d_\Phi)\) is Lipschitz. 

\begin{proposition}\label{prop:explicit-description-L-2}
Suppose Finsler volumes are measured  with respect to the circumscribed Riemannian volume functional. Then \(T:=\iota_\#\bb{M}\) is a rectifiable \(\bL^2\)-current in \((M, d_\Phi)\) and \(\norm{T}_{\bL^2}=\vol_\Phi\).
\end{proposition}

In \cite{ivanov-2009}, Ivanov proved that with respect to the inscribed Riemannian volume, Finsler filling volumes and Riemannian filling volumes agree. Hence, it seems worthwhile to investigate whether there exists a Lipschitz norm \(\bL\) such that Proposition~\ref{prop:explicit-description-L-2} is also valid for the inscribed Riemannian volume. 

Our second application deals with convexity properties of \(n\)-volume densities, a key object of study in Minkowski geometry with many applications in other fields. Let \(X\) be a finite-dimensional normed vector space of dimension strictly bigger than \(n\). We denote by \(\Lambda^n X\) the \(n\)-th exterior power of \(X\) and by \(\Lambda^n_s X \subset \Lambda^n X\) the cone of all simple \(n\)-vectors.  

\begin{definition}
A continuous map \(\phi \colon \Lambda^n_s X \to \R\) is called \textit{\(n\)-volume density} if \(\phi(\lambda a)=\abs{\lambda}\, \phi(a)\) for all \(a\in \Lambda_s^n X\) and all \(\lambda \in \R\), and it holds \(\phi(a)\geq 0\) with equality if and only if \(a=0\).    
\end{definition} 

For example, let \(\phi^b( v_1 \wedge \dotsm \wedge v_n)\) be equal to the \(\Haus^n_{\hspace{-0.25em}X}\)-measure of the parallelotope spanned by the vectors \(v_1, \ldots, v_n\in X\). This defines an \(n\)-volume density, called the Busemann \(n\)-volume density. An important open question, going back to Busemann, is whether \(\phi^b\) can always be extended to a norm on \(\Lambda^n X\). Busemann showed that this is possible in codimension one (i.e.~ if \(\dim X=n+1\)). Moreover, a breakthrough result of Burago and Ivanov \cite{burago-2012} gives an affirmative answer for \(n=2\). These results constitute the only known progress on this difficult question. See also \cite{2022arXiv220317160V} for a recent related  result. 

Any definition of volume gives rise to an \(n\)-volume density in a natural way. As it turns out, the analogues of Busemann's question have affirmative answers when volume densities induced by Gromov's mass*, Ivanov's inscribed Riemannian, or Bernig's definition of volume are considered (see \cite{thompson-2004, bernig-2014, ivanov-2009}). Our next result shows that this is also the case for the circumscribed Riemannian definition of volume. 

\begin{theorem}\label{prop:extendabily-convex}
For any finite-dimensional normed space \(X\), the \(n\)-volume density \(\phi\colon \Lambda^n_s X \to \R\) induced by the circumscribed Riemannian definition of volume is extendibly convex, that is, \(\phi\) is the restriction of a norm on \(\Lambda^n X\) to the cone of simple \(n\)-vectors.
\end{theorem}

To prove Theorem~\ref{prop:extendabily-convex} with the methods of \(\bL^2\)-currents we need to rely on a deep result of Burago and Ivanov \cite[Theorem~3]{ivanov-2004}. They showed that \(\phi\) is extendibly convex if and only if it is semi-elliptic over \(\R\). This notion goes back to Almgren \cite{almgren-1968} and states that when volumes are measured with respect to \(\phi\), any \(n\)-ball embedded into an \(n\)-dimensional affine subspace of \(X\) has minimal volume among all Lipschitz chains in \(X\) with real coefficients that have the same boundary. For Gromov's mass* volumes this is well-known (see \cite{benson-1966, thompson-1999, gromov-1983}). Ambrosio and Kirchheim \cite[Appendix C]{ambrosio-2000} gave an alternative proof of this fact by using the lower semicontinuity of the mass of a current. Essentially the same proof carries over to \(\bL^2\)-currents, proving Theorem~\ref{prop:extendabily-convex}. This is discussed in more detail in Section~\ref{section:six}.


\section{Basic properties of \(\bL\)-currents}
Metric currents enjoy many good functorality properties. For example, there are well-defined push-forward and restriction operators for them. To obtain analogous operations for \(\bL\)-currents, it is inevitable to impose some restrictions on the Lipschitz norms considered. Let \(X\) be a complete metric space and \(\bL\colon \Lip(X, \R^n) \to \R\) a Lipschitz norm. We say that \(\bL\) is \textit{compatible} if 
\[
\bL(\beta\circ \pi)\leq \Lip(\beta)\cdot \bL(\pi)
\]
for every \(\beta\in\Lip(\R^n, \R^n)\) and every \(\pi\in \Lip(X, \R^n)\). For example, the Lipschitz norms \(\bL^{\sigma}\) defined in \eqref{eq:main-example-Lip-norm} are compatible. Suppose from now on that \(\bL\) is compatible and \(T\) is an \(\bL\)-current of dimension \(n\geq 1\). For later use, we repeat in the following an argument due to Lang \cite{lang-2011} to show that \(T\) satisfies the following \textit{strict} locality property. 

Let \(f\in \Lip_b(X)\) and \(\pi\in \Lip(X, \R^n)\) be such that \(\pi^{(1)}=:\psi\) is constant on \(\spt(f)\). Then \(T(f, \pi)=0\). Indeed, we may suppose that \(\psi=0\) on \(\spt(f)\) and thus 
\[
\psi_{m}(x)=
\begin{cases}
0& \text{if } \abs{\psi(x)} \leq \frac{1}{m} \\
\psi(x)-\frac{1}{m} & \text{if } \psi(x)\geq \frac{1}{m} \\
\psi(x)+\frac{1}{m} & \text{if } \psi(x)\leq -\frac{1}{m} \\
\end{cases}
\]
is constant on an open neighborhood of \(\spt(f)\). Letting \(\pi_m=(\psi_m, \pi^{(2)}, \ldots, \pi^{(n)})\), this shows that \(T(f, \pi_m)=0\). By construction, \(\psi_m=\beta_m \circ \psi\) for a certain \(1\)-Lipschitz map \(\beta_m\colon \R^n \to \R^n\). Since \(\bL\) is compatible, we have \(\sup_m \bL(\pi_m) \leq \bL(\pi)\), and because of the continuity axiom \ref{it:two} it therefore follows that \(T(f, \pi_m)\) converges to \(T(f, \pi)\) as \(m\to \infty.\) This yields \(T(f, \pi)=0\), as desired.

We now discuss how to restrict \(\bL\)-currents to Borel subsets. Let \(\mathcal{B}_\infty(X)\) denote the set of all bounded real-valued Borel measurable functions on \(X\). Using the finite mass axiom \ref{it:three}, it is straightforward to show that \(T\) can be uniquely extended to a multilinear functional on tuples with first argument an element of \(\mathcal{B}_\infty(X)\). This follows directly from the 'folklore' result that \(\Lip_b(X, \R)\) is dense in \(L^1(X, \mu)\) for any Borel measure \(\mu\) on \(X\) (see e.g. \cite[Lemma~A.1]{hanneke-2021} for a detailed proof of this fact). By the above, for any Borel subset \(A\subset X\), 
\[
(T \on A)(f, \pi)=T(f\cdot \mathbbm{1}_A, \pi)
\]
defines an \(\bL\)-current on \(X\). Let us briefly explain why  \(T\on A\) satisfies the locality axiom \ref{it:two}. Suppose \(f\in \Lip_b(X, \R)\) and let \(\varphi_i\colon X \to \R\) be bounded Lipschitz functions such that \(\varphi_i \to \mathbbm{1}_A\) in \(L^1(X, \norm{T})\) as \(i\to \infty\). Hence, 
\[
(T\on A)(f, \pi)=\lim_{i\to \infty} T(f\cdot \varphi_i, \pi).
\]
Since \(\spt(f\cdot \varphi_i)\subset \spt(f)\), it follows form the strict locality property established above that if \(\pi\in \Lip(X, \R^n)\) is such that \(\pi^{(1)}\) is constant on \(\spt(f)\), then \(T(f\cdot \varphi_i, \pi)=0\) and so \((T\on A)(f, \pi)=0\). This shows that \(T\on A\) also has the strict locality property. 

By definition, \(T\on A\) is an \(\bL\)-current defined on \(\mathcal{D}^n(X)\). Sometimes it is convenient to consider such currents as currents in \(A\). Suppose \(A\subset X\) is a closed subset and write \(\bL_A\colon \Lip(A, \R^n) \to \R\) for the quotient semi-norm
\[
\bL_A(\pi)=\inf\big\{ \bL(\tilde{\pi}) \, \mid \, \tilde{\pi}\in \Lip(X, \R^n) \text{ with } \tilde{\pi}|_{A}=\pi \big\}.
\]
Furthermore, for \((f, \pi)\in \mathcal{D}^n(A)\) we define
\[
T_A(f, \pi)=(T\on A)(\tilde{f}, \tilde{\pi}),
\]
where \(\tilde{f}\in \Lip_b(X, \R)\) is a Lipschitz extension of \(f\) and \(\tilde{\pi}\in \Lip(X, \R^n)\) a Lipschitz extension of \(\pi\). By the strict locality property established above, this defines a well-defined \(\bL_A\)-current on \(A\). The existence of a bounded Lipschitz extension of \(f\) follows directly from McShane's extension theorem (see e.g. \cite[Theorem~ 1.33]{weaver-2018}). We claim that \(\norm{T_A}_{\bL_A} \leq \norm{T}_{\bL}\on A\). This can be seen as follows. We abbreviate \(\mu=\norm{T}_{\bL}\) and let \(U\subset X\) be an open subset containing \(A\) such that \(\mu(U\setminus A) \leq \epsilon\). The function \(\varphi\colon X \to [0,1]\) defined by
\[
\varphi(x)=\frac{d(x, X\setminus U)}{d(x, X\setminus U)+d(x, A)}
\]
is continuous,  \(\varphi(x)=1\) for all \(x\in A\) and \(\varphi(x)=0\) for all \(x\in X\setminus U\). Hence, if \(\tilde{f}\) is a Lipschitz extension of \(f\) with \(\norm{\tilde{f}}_\infty=\norm{f}_{\infty}\), then letting \(C=\big[\bL(\tilde{\pi})\big]^n\), we find that
\[
T(\varphi\cdot \tilde{f}, \tilde{\pi})\leq C \int_X \varphi\cdot \tilde{f} \, d\mu \leq C\int_{A} f\, d\mu+C\cdot\norm{f}_\infty\cdot \epsilon.
\]
This shows that \(\norm{T_A}_{\bL_A}\leq \mu\on A\), as desired. 

The rest of this section is about push-forwards of currents. Let \(\varphi\colon X \to Y\) be a Lipschitz map between complete metric spaces. We consider the semi-norm \(\varphi_\# \bL\colon \Lip(Y,  
 \R^n)\to \R\) defined by \(
(\varphi_\#\bL)(\pi)=\bL(\pi \circ \varphi)\).  It is easy to check that setting
\[
\varphi_{\#}T(f, \pi)=T(f\circ \varphi, \pi \circ \varphi)
\]
for all \((f, \pi)\in \mathcal{D}^n(Y)\) defines a \(\varphi_\# \bL\)-current in \(Y\). Notice that 
\[
\norm{\varphi_\# T}_{\varphi_\# \bL} \leq \varphi_\# \norm{T}_{\bL} \quad \quad \text{ and } \quad \quad \norm{\varphi_\# T}_{\bL} \leq \big[\Lip(\varphi)\big]^n \cdot\, \varphi_\#\norm{T}_{\bL}
\]
if \(\bL\) is as in \eqref{eq:main-example-Lip-norm}. In practice, one often encounters Lipschitz maps which are defined on a closed subset \(A\subset X\) such that \(\spt T\) is contained in \(A\). For such currents, we occasionally write \(\varphi_\# T\coloneqq \varphi_{\#} (T_A)\) when readability demands it. 


\section{Rectifiable \(\bL\)-currents}\label{sec:three}

Important examples of metric currents are \(n\)-currents in \(\R^n\) induced by \(L^1\)-functions. Suppose \(\theta\in L^1(\R^n)\), then 
\[
\bb{\theta}(f, \pi)\coloneqq \int_{\R^n} \theta \cdot f\cdot \det(D\pi) \,\, d\mathscr{L}^n
\]
induces a multilinear functional on \(\mathcal{D}^n(\R^n)\). The existence of \(D\pi\) at almost every point is guaranteed by Rademacher's theorem. By \cite[Theorems~8.8 and 5.46]{dacorogna-2008}, the function \(I(u)=\int_{\Omega} \theta\cdot f \cdot \det(Du)\, d\mathscr{L}^n\) is weak-\(\ast\) continuous in \(W^{1, \infty}(\Omega)\) for every open subset \(\Omega\subset \R^n\) with Lipschitz boundary. Hence, it is not difficult to show that \(\bb{\theta}\) defines an \(n\)-current in \(\R^n\); see \cite[Example~3.2]{ambrosio-2000}. Push-forwards of these currents can be used as building blocks to obtain a parametric representation of rectifiable currents in general metric spaces (see \cite[Theorem~4.5]{ambrosio-2000}). 

The main result of this section shows that \(\bL\)-currents also admit such a parametric representation.


\begin{proposition}\label{prop:parametric}
Let \(\bL\) be a compatible Lipschitz norm and \(T\) an \(n\)-dimensional rectifiable \(\bL\)-current which is concentrated on the \(n\)-rectifiable set \(S\subset X\). Let \((K_i, \varphi_i)\) be a bi-Lipschitz parametrization of \(S\). Then there exist functions \(\theta_i\in L^1(\R^n)\) with \(\spt \theta_i\subset K_i\) such that
\begin{equation}\label{eq:parametric-representation}
T=\sum_{i=1}^\infty \varphi_{i\#}\big(\bb{\theta_i}\big) \quad \quad \text{ and } \quad \quad \norm{T}_{\bL}=\sum_{i=1}^\infty \norm{\, \varphi_{i\#}\big(\bb{\theta_i}\big) \,}_{\bL},
\end{equation}
where we use the notation \(\varphi_{\#}\big(\bb{\theta}\big):=\varphi_{\#}\big(\bb{\theta}_{K}\big)\).
\end{proposition}

Strictly speaking, for the proof of our main results from the introduction, we only need that \eqref{eq:parametric-representation} holds for \(\bL^\sigma\)-currents. And for these Lipschitz norms, the proof of the proposition could be simplified considerably. However, we believe that compatible Lipschitz norms could be quite useful in future investigations and therefore decided to formulate  Proposition~\ref{prop:parametric} in a more general way than would actually be necessary for the purpose of this paper. 

\begin{proof}[Proof of Proposition~\ref{prop:parametric}]
We set \(\mu=\norm{T}_{\bL}\) and \(A_i=\varphi_i(K_i)\). For every \(i\geq 1\), 
\begin{equation}\label{eq:representation-of-T}
T\big(f-(\mathbbm{1}_{A_1}f+\dotsm+\mathbbm{1}_{A_N}f), \pi\big)\leq C\cdot \mu\Big(S \setminus \bigcup_{i=1}^N A_i\Big),
\end{equation}
where \(C=\bL(\pi)^n \cdot \norm{f}_\infty\). By assumption, \(\mu\) vanishes on \(\Haus^n\)-negligible Borel sets. Since for every \(\Haus^n\)-measurable set \(A\subset X\), there is a Borel set \(B\subset X\) with \(\Haus^n(A)=\Haus^n(B)\), we may conclude that \(\mu\) also vanishes on every \(\Haus^n\)-measurable \(\Haus^n\)-negligible set. It follows that the left-hand side of \eqref{eq:representation-of-T} converges to zero as \(N\to \infty\) and so \(T\) admits the representation
\begin{equation}\label{eq:aux-representation}
T(f, \pi)=\lim_{N\to \infty} \sum_{i=1}^N T(\mathbbm{1}_{A_i} \cdot f, \pi).
\end{equation}
Let \(\psi_i\colon A_i\to \R^n\) denote the inverse of \(\varphi_i\) and set \(\bL_i:=\psi_{i\#}\bL_{A_i}\).
We now consider the \(\bL_i\)-current \(T_i:=\psi_{i\#} T_{A_i}\) and show that \(\norm{T_i}_{\bL_i}\) satisfies the assumptions of Lemma~\ref{lem:main-technical-lemma} below. To make the arguments more readable, from now on we omit the indices indicating the Lipschitz norm in measures of the form \(\norm{T}_{\bL}\). 

Since \(\norm{T_i} \leq \psi_{i\#} \norm{T_{A_i}}\), it follows that \(\norm{T_i}\) is concentrated on \(K_i\). Next, we show that \(\norm{T_i}\) vanishes on \(\Haus^n\)-negligible Borel sets. To this end, let \(A\subset \R^n\) be a Borel set. If \(\Haus^n(A)=0\), then \(\Haus^n(A\cap K_i)=0\) and since \(\varphi_i\) is Lipschitz, \(\Haus^n(\psi_i^{-1}(A\cap K_i))=0\) as well.  We know that \(\norm{T_{A_i}} \leq \norm{T}\on A_i\). Hence, by putting everything together, we find that 
\[
\norm{T_i}(A)=\norm{T_i}(A\cap K_i)\leq \norm{T_{A_i}}\big(\psi_i^{-1}(A\cap K_i)\big)\leq \norm{T}\big(\psi_i^{-1}(A\cap K_i)\big) \leq 0,
\]
where for the last inequality we used that \(T\) is rectifiable. Now,  Lemma~\ref{lem:main-technical-lemma} tells us that 
\[
\psi_{i\#} T_{A_i}=\bb{\theta_i}
\]
for some \(\theta_i\in L^1(\R^n)\). Since \(\norm{T_i}\) is concentrated on \(K_i\), it follows that \(\spt \theta_i\subset K_i\). Moreover, \(T_{A_i}=\big(\varphi_{i\#}\bb{\theta_i}\big)_{A_i}\) and so the left equality in \eqref{eq:parametric-representation} follows from \eqref{eq:aux-representation}.

To finish the proof, we now show the other equality in \eqref{eq:parametric-representation}. Letting \(\mu_i=\norm{T_i}_{\bL}\) for \(T_i=\varphi_{i\#}\big(\bb{\theta_i}\big)\), we get
\begin{equation}\label{eq:aux-2}
\abs{T(f, \pi)}\leq \sum_{i=1}^\infty \int_{X} \abs{f} \, d\mu_i
\end{equation}
for all \((f, \pi)\in \mathcal{D}^n(X)\) with \(\bL(\pi)\leq 1\). Because of \(\mu_i \leq \norm{T}_{\bL} \on \varphi_i(K_i)\), 
the infinite series
\[
\tilde{\mu}(A):=\sum_{i=1}^\infty \mu_i(A)
\]
converges for every Borel set \(A\subset X\). This shows that \(\tilde{\mu}\) defines a finite Borel measure on \(X\) such that \(\tilde{\mu}\leq \norm{T}_{\bL}\). Because of \eqref{eq:aux-2}, it follows that \(\tilde{\mu}=\norm{T}_{\bL}\), as desired.
\end{proof}

Thus, to finish the proof of Proposition~\ref{prop:parametric}, we need to establish the following lemma. 

\begin{lemma}\label{lem:main-technical-lemma}
Let \(\bL\) be a compatible Lipschitz norm on \(\R^n\). Suppose \(T\) is an \(n\)-dimensional \(\bL\)-current in \(\R^n\) with the following properties:
\begin{enumerate}
    \item \(\norm{T}_{\bL}\) vanishes on \(\Haus^n\)-negligible Borel sets,
    \item\label{it:concentrated} and \(\norm{T}_{\bL}\)  is concentrated on a compact set.
\end{enumerate}
Then there exists \(\theta\in L^1(\R^n)\) such that \(T=\bb{\theta}\), that is,
\begin{equation}\label{eq:rep-formula}
T(f, \pi)=\int_{\R^n} \theta \cdot f\cdot \det(D\pi) \,\, d\mathscr{L}^n
\end{equation}
for all \((f, \pi)\in \mathcal{D}^n(\R^n)\). In particular, \(T\) is an \(n\)-current.
\end{lemma}

\begin{proof}
To begin, we show that \(T\) can be interpreted as a local current in the sense of Lang \cite[Definition~2.1]{lang-2011}. 
We may suppose that \(\norm{T}_{\bL}\) is concentrated on the Euclidean \(n\)-ball \(B\subset \R^n\). Let \(I\colon B \to \R^n\) denote the inclusion map. Consider a Lipschitz map \(f\in \Lip(\R^n, \R)\) with compact support and suppose \(\pi=(\pi^{(1)}, \ldots, \pi^{(n)})\) with \(\pi^{(i)}\in \Lip_{\loc}(\R^n, \R)\). Then \((f\circ I, \pi \circ I)\in \mathcal{D}^n(B)\)
and so 
\[
T^\ast(f, \pi):=T_B(f\circ I, \pi \circ I)
\]
is a well-defined multilinear functional on \(\Lip_c(\R^n, \R)\times \bigl[\Lip_{\loc}(\R^n, \R)\bigr]^n\). For brevity, we denote the latter set by \(\mathcal{D}_{\loc}^n(
\R^n)\). 

Let \(f_i\in \Lip_c(\R^n, \R)\) be compactly supported  Lipschitz functions such that \(\sup_i \Lip(f_i) \leq C\) for some uniform constant and there is a compact subset \(K\subset \R^n\) with \(\spt(f_i)\subset K\). Further, suppose \(\pi_i\in \bigl[\Lip_{\loc}(\R^n, \R)\bigr]^n\) are such that for every compact \(K\subset \R^n\) there is \(C_K>0\) so that \(\sup_i \Lip(\pi_i|_{K}) \leq C_K\).  We claim that if we have pointwise convergences \(f_i \to f\) and \(\pi_i \to \pi\) with \((f, \pi)\in \mathcal{D}_{\loc}^n(\R^n)\), then \(
T^\ast(f, \pi)=\lim_{i\to \infty} T^\ast(f_i, \pi_i).
\)
Clearly \(\bL_B(\pi_i \circ I)\leq C_B \cdot \bL_B(\id_B)\) and  thus, letting \(C=C_B \cdot \bL_B(\id_B)\), we infer
\[
\abs{T^\ast(f-f_i, \pi_i)}\leq C \int_{B} \abs{f-f_i} \, d\norm{T_B}_{\bL} \leq C' \cdot \norm{(f-f_i) \circ I}_{\infty},
\]
where \(C'= C^n\cdot \mass_{\bL}(T_B)\). This implies that
\begin{align*}
\lim_{i\to \infty} T^\ast(f_i, \pi_i)&=\lim_{i\to \infty} T^\ast(f, \pi_i) =\lim_{i\to \infty} T_B(f\circ I, \pi_i\circ I)\\
&=T_B(f\circ I, \pi\circ I)=T^\ast(f, \pi),
\end{align*}
which proves the continuity of \(T^\ast\). For the locality property, suppose some \(\pi^{(i)}\) is constant on an open neighborhood of \(\spt(f)\). Since \(\spt(f\circ I)\subset \spt(f) \cap B\), the composition \(\pi^{(i)}\circ I\) is constant on \(\spt(f\circ I)\). Hence, we have 
\[
T^\ast(f, \pi)=T_B(f\circ I, \pi\circ I)=0,
\]
as \(T_B\) satisfies the strict locality property.

From the above, we conclude that \(T^\ast\) is a local \(n\)-current in \(\R^n\). In particular, if \(\psi=(\psi^{(1)}, \ldots, \psi^{(n)})\in \big[C^{1,1}(\R^n)\big]^n\), then
\[
T^\ast(f, \psi)=T^\ast(f\cdot \det D\psi, \id_{\R^n})
\]
for all \(f\in \Lip_c(\R^n, \R)\); see \cite[Theorem~2.5]{lang-2011}. Here, we follow Lang and use \(C^{1,1}(\R^n)\) to denote the set of all \(\psi\in C^1(\R^n)\) with partial derivatives \(D_1 \psi, \ldots, D_n \psi \in \Lip_{\loc}(\R^n)\). We now consider the linear functional \(\Lambda \colon C_c(\R^n) \to \R\) defined by
\[
\Lambda(f)=T_B(f\circ I, \id_B).
\]
Because of \((f\circ I)\in \mathcal{B}_\infty(B)\),  this is well-defined. Moreover, since \(T_B\) satisfies the finite mass axiom, 
\[
\abs{\Lambda(f)}\leq \Big[\bL_B(\id_B)^n \mass(T_B)\Big]\cdot \norm{f}_\infty
\]
for all \(f\in C_c(\R^n)\).
By the Riesz representation theorem, there exists a signed Borel measure \(\eta\) on \(\R^n\) such that 
\[
\Lambda(f)=\int_{\R^n} f \, d\eta.
\] 
Thus, for all \((f, \pi)\in \mathcal{D}^n(\R^n)\) with \(\spt(f)\) compact and \(\pi\in \big[C^{1,1}(\R^n)\big]^n\), it holds
\begin{equation}\label{eq:very-important-equality}
T(f, \pi)=T^\ast(f, \pi)=T^\ast(f \cdot \det D\pi, \id_{\R^n})=\int_{\R^n} f \cdot \det(D\pi) \, d\eta.
\end{equation}
Consequently, by a standard approximation argument, it follows that \eqref{eq:very-important-equality} is valid for all \((f, \pi)\in \mathcal{D}^n(\R^n)\). Notice that if \(Q=I_1 \times \dotsm \times I_n\) is a product of intervals \(I_i\subset \R\), then there exist a sequence \((f_i)\) of Lipschitz functions \(\R^n \to \R\) such that \(\abs{f_i} \leq 1\) and \(f_i \to \mathbbm{1}_Q\) pointwise. In particular, using \eqref{eq:very-important-equality}, we find that
\[
T(\mathbbm{1}_Q, \id_{\R^n})=\lim_{i\to \infty} T(f_i, \id_{\R^n})=\lim_{i\to \infty} \int_{\R^n} f_i \, d\eta=\eta(Q)
\]
by the dominated convergence theorem. This implies that \(\eta(A)=T(\mathbbm{1}_A, \id_{\R^n})\) for every Borel set \(A\subset \R^n\). Clearly, \(\eta(A)\leq \bL(\id_{\R^n})^n \int_A 1 \, d\norm{T}_{\bL}\), and so by our assumption (\ref{it:concentrated}) on \(\norm{T}_{\bL}\), it follows that \(\eta\) is absolutely continuous with respect to the Lebesgue measure on \(\R^n\). An application of the Radon-Nikodým theorem now yields a function \(\theta\in L^1(\R^n)\) such that \(\int_{\R^n} g\, d\eta=\int_{\R^n} \theta \cdot g \, d\mathscr{L}^n\) for every \(\abs{\eta}\)-integrable function \(g\colon \R^n \to \R\). Therefore, \eqref{eq:rep-formula} follows from \eqref{eq:very-important-equality}.
\end{proof}

\section{Jacobians and volumes}\label{sec:four}

To give a precise description of the measures appearing in the parametric representation \eqref{eq:parametric-representation}, it is necessary to introduce tools from Finsler geometry. 

\subsection{Jacobians}
A map \(\J\colon \Sigma \to \R_{\geq 0}\) defined on the set \(\Sigma\) of all semi-norms on \(\R^n\) is called a \textit{normalized Jacobian} if the following holds: 
\begin{enumerate}
    \item\label{it:jac-one} \(\J(\abs{\,\cdot\,})=1\) for the standard Euclidean norm \(\abs{\,\cdot\,}\),
    \item\label{it:jac-two} \(\J(s)\leq \J(s')\) whenever \(s \leq s'\),
    \item\label{it:jac-three} if \(s=s' \circ L\) for a bijective linear map \(L\colon \R^n \to \R^n\), then \(\J(s)=\J(s')\cdot \abs{\det(L)}\). 
\end{enumerate}
For our purposes it is often more natural not to impose condition \eqref{it:jac-one}. Thus, any map \(\J\colon \Sigma \to \R_{\geq 0}\) satisfying \eqref{it:jac-two} and \eqref{it:jac-three} will be called \textit{Jacobian} in this work. For example, if \(\sigma\) is a fixed norm on \(\R^n\) with unit ball \(B_\sigma\), then the following map defines a Jacobian. For every norm \(s\in \Sigma\) we set
\begin{equation}\label{eq:definition-jacobian}
\J^{\sigma}(s)=\sup_{F} \frac{\Leb^n(B_\sigma)}{\mathscr{L}^n(F^{-1}(B_\sigma))},
\end{equation}
where the supremum is taken over the set \(\mathcal{F}\) of all linear maps \(F\colon \R^n \to \R^n\) such that \(F^{-1}(B_\sigma)\) contains the unit ball of \(s\). For any other \(s\in \Sigma\) which is not a norm, we define \(\J^\sigma(s)=0\). A straightforward computation shows that \(\J^\sigma(\abs{\,\cdot\,})=\J^{\ir}(\sigma)\). In particular, \(\J^\sigma\) is a normalized Jacobian if \(\J^{\ir}(\sigma)=1\). Here, \(\J^{\ir}\) denotes Ivanov's inscribed Riemannian Jacobian.

The following lemma shows that the supremum in \eqref{eq:definition-jacobian} is attained. However, this is not too important for the main goal of this paper. We mainly included it to make some of the proofs in Section~\ref{sec:section-5} a little less technical.

\begin{lemma}\label{lem:not-important}
There exists \(F\colon \R^n \to \R^n\) such that the supremum in  \eqref{eq:definition-jacobian} is achieved.  
\end{lemma}

\begin{proof}
Since \(s\) is a norm on \(\R^n\) there exists \(\epsilon>0\) such that the closed Euclidean ball \(B(0, \epsilon)\) is contained in \(B_s\). Let \(F\in \mathcal{F}\). If \(F^{-1}(B_\sigma)\) is not contained in \(B(0, R)\), then
\[
\Leb^n(F^{-1}(B_\sigma)) \geq \Big( \frac{\epsilon}{2}\Big)^{n-1} \cdot R.
\]
This implies that there are \(R_0>0\) and a minimizing sequence \((F_i)\subset \mathcal{F}\) such that \(F_i^{-1}(B_\sigma) \subset B(0, R_0)\) for every \(i\geq 1\). In particular, each \(F_i\) is bijective and \(F_i^{-1}\circ I\) is \(R_0\)-Lipschitz, where \(I\colon (\R^n , \sigma) \to (\R^n, \abs{\cdot})\) denotes the identity map. By a standard compactness argument, there exists a subsequence, also denoted by \((F_i)\), that converges to some bijective \(F\in \mathcal{F}\) such that \(\Jac(F_i^{-1}) \to \Jac(F^{-1})\) as \(i\to \infty\). Hence, it follows from the area formula that \(\Leb^n(F^{-1}(B_\sigma))=\lim_{i\to \infty} \Leb^n(F^{-1}_i(B_\sigma))\). In particular, the supremum in \eqref{eq:definition-jacobian} is achieved by \(F\).
\end{proof}

Other important examples are the Busemann and Holmes-Thompson Jacobians, denoted by \(\J^{\bus}\) and \(\J^{\htt}\), respectively. We refer to \cite[Section~3.1]{thompson-2004} for their definitions and further information. 

Most importantly, Jacobians can be used to associate Finsler volumes to \(n\)-rectifiable metric spaces. To make this precise, we recall in the following some basic notions from metric analysis.

\subsection{Metric differentials and related notions}
It is a well-known consequence of Rademacher's theorem that any Lipschitz map \(\varphi\colon A \to \R^m\), where \(A\subset \R^n\) a Borel subset, is differentiable almost everywhere. This means that for almost every \(p\in A\) there exists a linear map \(D_p \varphi\colon \R^n \to \R^m\) such that
\[
\lim_{q\in A, \,q\to p}\frac{\varphi(q)-\varphi(p)-(D_p \varphi)(q-p)}{\abs{q-p}}=0.
\]
If no such map exists we use the convention that \(D_p \varphi\) denotes the zero map. In \cite{kirchheim-1994}, Kirchheim generalized Rademacher's theorem to Lipschitz maps \(\varphi\colon A \to X\).
He showed that for almost every \(p\in A\) there exists a seminorm \(\md_p \varphi\) on \(\R^n\) such that
\[
\lim_{q\in A, \,q\to p} \frac{d(\varphi(p), \varphi(q))-(\md_p \varphi)(q-p)}{\abs{q-p}}=0.
\]
Suppose now that \(X\subset \ell_\infty\), which is always possible if \(X\) is separable. For each \(j\in \N\) let \(\text{proj}_j \colon \ell_\infty \to \R\) denote the projection onto the \(j\)-th coordinate. It can be shown that if \(f\colon A\to X\) is Lipschitz, then for almost every \(p\in A\), there exists a linear map \(\weakd_p \varphi\colon \R^n \to \ell_\infty\) such that
\begin{equation}\label{eq:relation-weak-differential-and-metric-diff}
(\md_p\varphi)(\cdot)=\norm{(\weakd_p\varphi)(\cdot)}_\infty.
\end{equation}
and 
\[
D_p\big( \text{proj}_j \circ \varphi\big)=\text{proj}_j \circ \weakd_p\varphi
\]
for every \(j\in \N\). See, for example, \cite{ambrosio--rectifiable-2000, kirchheim-1994, ivanov-2009, korevaar-1993} for a proof of this. The map \(\weakd_p \varphi\) is called the \(\text{w}^\ast\)-differential of \(\varphi\) at \(p\). 

An important application of metric differentials is the following metric area formula due to Kirchheim \cite{kirchheim-1994},
\begin{equation}\label{eq:area-formula}
\int_A \theta\big(\varphi(p)\big) \, \J^{\bus}(\md_p \varphi) \, dp=\int_X \theta(x)\,\, \#\big(A\cap \varphi^{-1}(x)\big) \, d\hspace{-0.2em}\Haus^n(x)
\end{equation}
for any Borel function \(\theta \colon X \to \R\). Here, \(\J^{\bus}\) denotes the Busemann Jacobian. If \(F\colon V \to W\) is a linear map between Banach spaces so that \(\dim V=n\), then for any Borel set \(A\subset V\) with \(\Haus^n(A)>0\), 
\begin{equation}\label{eq:def-jacobian}
\Jac(F)=\frac{\Haus^n(F(A))}{\Haus^n(A)},
\end{equation}
is a well-defined real-number called the \textit{Jacobian} of \(F\). Due to \eqref{eq:relation-weak-differential-and-metric-diff}, it follows that
\[
\J^{\bus}(\md_p \varphi)=\Jac(\weakd_p\varphi),
\]
for \(\Haus^n\)-almost every \(p\in A\). In particular, if \(X=\R^n\), then \(\weakd_p \varphi=D_p \, \varphi\) and \(\Jac( D_p \,\varphi)=\abs{\det D_p \,\varphi}\) almost everywhere, and thus \eqref{eq:area-formula} reduces to the classical area formula in this case.

\subsection{Finsler volumes}\label{sec:finsler-vols} Let \(\J^\bullet\) be a Jacobian and \(S\subset X\) an \(n\)-rectifiable subset of a complete metric space. Suppose that \((K_i, \varphi_i)\) is a bi-Lipschitz parametrization of \(S\). Then
\[
\mu_S^\bullet(A)=\sum_{i=1}^\infty \int_{\varphi_i^{-1}(A)} \J^\bullet(\md_p \varphi_i)\, dp
\]
defines a Borel measure on \(X\). Using the chain rule for metric differentials and \eqref{it:jac-three} in the definition of Jacobians, it follows that \(\mu_S^\bullet\) does not depend on the choice of bi-Lipschitz parametrization \((K_i, \varphi_i)\). The measure \(\mu_S^\bullet\) is called \textit{Finsler volume} associated to \(S\). Clearly, every \(n\)-dimensional Banach space \(V\) is \(n\)-rectifiable. It follows directly from the construction that \(\mu_V^\bullet\) is a Haar measure on \(V\). An assignment \(V \mapsto \vol_V\) which assigns to each \(n\)-dimensional normed space \(V\) a Haar measure \(\vol_V\) on \(V\) such that
\begin{itemize}
    \item[(*)] If \(F\colon V \to W\) is linear and \(\norm{F} \leq 1\), then \(\vol_W(F(A))\leq \vol_V(A)\) for all Borel sets \(A\subset V\). In other words, any such map \(F\) is volume non-increasing.
\end{itemize}
is called \textit{n-dimensional Banach volume functional}. It is easy to check that the assignment \(V \mapsto \mu_V^\bullet\) satisfies this definition. Conversely, if \( V \mapsto \vol_V\) is a Banach volume functional, then one can define a Jacobian \(\J\) as follows.
If \(s\in \Sigma\) is not a norm, we set \(\J(s)=0\), and
\[
\J(s)=\frac{\vol_{(\R^n, s)}(A)}{\Leb^n(A)}, \quad \quad \text{where \(A\subset \R^n\) is a Borel subset with \(\Leb^n(A)>0\)},
\]
otherwise. Using \cite[Lemma~4]{kirchheim-1994}, it is not difficult to show that the functors defined above establish a one-to-one correspondence between Banach volume functionals and Jacobians.

\section{Volumes of rectifiable \(\bL^{\sigma}\)-currents}\label{sec:section-5}

The aim of this section is to prove the following explicit description of the mass measure of \(\bL^\sigma\)-currents. 

\begin{theorem}\label{thm:formula-for-the-mass-measure}
Suppose \(T\) is a rectifiable \(\bL^\sigma\)-current which is concentrated on the \(n\)-rectifiable set \(S\subset X\). Let \((K_i, \varphi_i)\) be a bi-Lipschitz parametrization of \(S\) and let \(\theta_i\in L^1(\R^n)\) with \(\spt \theta_i \subset K_i\) be such that the parametric representation \eqref{eq:parametric-representation} holds. Then
\[
\norm{T}_{\bL^\sigma}=\abs{\theta}\, d\mu^\sigma_S,
\]
where \(\theta\colon X \to \R\) is defined by \(\theta(x)=\theta_i(p)\) if \(x=\varphi_i(p)\) and \(\theta(x)=0\) otherwise.
\end{theorem}

Recall that the Lipschitz norm \(\bL^{\sigma}\) is defined by \eqref{eq:main-example-Lip-norm} for some fixed norm \(\sigma\) on \(\R^n\). Clearly, Theorem~\ref{thm:main-2} from the introduction follows directly from the theorem above. 

\subsection{Special case} We will prove Theorem~\ref{thm:formula-for-the-mass-measure} by reducing it to the special case \(T=\bb{\theta}\). To simplify the notation we will use the notation \(\bL=\bL^{\sigma}\) throughout this subsection.

\begin{lemma}\label{lem:special-case}
Let \(V=(\R^n, \norm{\cdot})\) be a normed space and denote by \(I\colon \R^n \to V\) the identity map. Then for every \(\theta\in L^1(\R^n)\), 
\[
\norm{T}_{\bL}=\abs{\theta} \, \mu^{\sigma}_V,
\]
where \(T=I_\#\big(\bb{\theta}\big)\).
\end{lemma}

\begin{proof}
For all \((f, \pi)\in \mathcal{D}^n(V)\), we have
\[
\abs{T(f, \pi)} \leq \int_{\R^n} \abs{f(p)\cdot \theta(p) \cdot \det D_p (\pi \circ I)}\, dp.
\] 
We set \(\tilde{\pi}=[\bL(\pi)]^{-1}\cdot \pi\). Since 
\[
\det D_p (\pi \circ I)=\det D_p (\tilde{\pi} \circ I)\cdot [\bL(\pi)]^{n},
\]
it follows that
\[
\norm{T}_{\bL}\leq \abs{\theta}\cdot \Big[\sup_{\pi\in \Pi_1} \Jac\big(D_{x} \,\pi\big)\Big] \, d\hspace{-0.2em}\Haus^n_V,
\]
where \(\Pi_1\) denotes the set of all \(\pi\in \Lip(V, \R^n)\) with \(\bL(\pi)\leq 1\). The Jacobian \(\Jac\) of a linear map is defined as in \eqref{eq:def-jacobian}. By virtue of Lemma~\ref{lem:desired-inequality} below,
\[
 \Big[\sup_{\pi\in \Pi_1} \Jac\big(D_{x}\, \pi\big)\Big] \, d\hspace{-0.2em}\Haus^n_V=  d\mu^{\sigma}_V,
\]
Hence, 
\begin{equation}\label{eq:intermediate-step-1}
\norm{T}_{\bL} \leq \abs{\theta}\, d\mu^\sigma_V.
\end{equation}
In the following, we show that the reverse inequality also holds. Fix  \(\pi\in \Pi_1\). The signed measure \(\eta\colon \mathcal{B}(V) \to \R\) defined by
\[
\eta(A)=\sup\Big\{ T(f, \pi) : f\in \mathcal{B}_\infty(V) \text{ with } 
 \abs{f} \leq \mathbbm{1}_A\Big\}
\]
satisfies \(\eta \leq \norm{T}_{\bL}\) by construction. As the map \(\R^n \ni p \mapsto \sgn\det D_p(\pi \circ I)\) is Borel measurable, 
\[
\eta(A)=\int_{A} \abs{\theta}\cdot \abs{\det D_p(\pi \circ I)} \, dp,
\]
and so we find that \(\norm{T}_{\bL} \geq \abs{\theta} \cdot \Jac(D_x \pi) \, d\hspace{-0.2em}\Haus^n_V\). Because of Lemma~\ref{lem:desired-inequality}, we know that there exists \(\pi^\ast\in \Pi_1\) such that
\[
\Jac( D_x \,\pi^\ast)=\sup_{\pi\in \Pi_1} \Jac\big(D_{x} \,\pi\big).
\]
It follows that \(\norm{T}_{\bL}\geq \abs{\theta}\, d\mu_V^\sigma\), as desired.
\end{proof}

The proof of the following result is straightforward. Basically, it is just a matter of combining the definition of the Jacobian \(\J^\sigma\) with the definition of the Lipschitz norm \(\bL=\bL^\sigma\).

\begin{lemma}\label{lem:desired-inequality}
Let \(V=(\R^n, \norm{\cdot})\) be a normed space and denote by \(I\colon \R^n \to V\) the identity map. Then for every \(\pi\in \Lip(V, \R^n)\) with \(\bL(\pi) \leq 1\),
\begin{equation}\label{eq:desired-inequality}
\Jac\bigl(D_{x}\,\pi\bigr)\cdot \Jac(D_p I)  \leq \J^{\sigma}(\md_p I) 
\end{equation}
 for \(\Haus^n\)-almost every \(p\in \R^n\). Moreover, there is a \(\pi^\ast\) for which this inequality becomes an equality almost everywhere.
\end{lemma}

We remark here that the left-hand side of \eqref{eq:desired-inequality} is equal to \(\abs{\det D_p (\pi \circ I)}\). Moreover, it clearly holds that \(\norm{\cdot}_V=\md_p I\), and in view of the proof of Lemma~\ref{lem:special-case} it is also useful to recall that \(\Jac(D_p \,I)=\J^{\bus}( \md_p I)\).

\begin{proof}
We abbreviate \(F=D_p (\pi\circ I)\). We consider the following commutative diagram   
\[
  \begin{tikzcd}
    \R^n \arrow{r}{F} \arrow[swap]{dr}{D_p I} & \R^n  \\
     & V \arrow{u}[right]{D_x\, \pi}.
  \end{tikzcd}
\]
Notice that if \(\R^n\) is equipped with \(\sigma\), then \(\norm{ D_x \,\pi}\leq 1\) due to our assumption that \(\bL(\pi)\leq 1\). This implies that the unit ball of \((\R^n, \md_p I)\) is contained in \(F^{-1}(B_\sigma)\).
As
\[
\Jac\bigl(D_{x}\,\pi\bigr)\cdot \Jac(D_p I)=\Jac(F)=\frac{\Leb^n(B_\sigma)}{\Leb^n(F^{-1}(B_\sigma))},
\]
\eqref{eq:desired-inequality} follows. Suppose now that \(F\colon \R^n \to \R^n\) is a linear map such that \(F^{-1}(B_\sigma)\) contains the unit ball of \((\R^n, \md_p I)\) and 
\[
\J^\sigma( \md_p I)=\frac{\Leb^n(B_\sigma)}{\Leb^n(F^{-1}(B_\sigma))}.
\]
The existence of such a map is guaranteed by Lemma~\ref{lem:not-important}. Clearly, \(\pi^\ast=F\circ I^{-1}\) has the desired properties. 
\end{proof}

\subsection{General case}
We now give the proof of Theorem~\ref{thm:formula-for-the-mass-measure}. The following argument is very similar to the proof of \cite[Lemma~2.5]{zust2021riemannian} due to Züst. Let \(L>1\) be arbitrary. Fix \(i\in \N\) and abbreviate \(K=K_i\), and \(\varphi=\varphi_i\) and \(\theta=\theta_i\). By a result of Kirchheim \cite[Lemma~4]{kirchheim-1994}, up to a set of \(\Haus^n\)-measure zero, \(K\) admits a decomposition \((K_{j})\) into compact sets such that the following holds. There are normed spaces \(V_j=(\R^n, \norm{\cdot}_j)\)  so that \(\varphi\) restricted to \(K_{j}\) admits a factorization \(\psi_j\circ (I_j|_{K_j})\), where, as always, \(I_j\colon \R^n \to V_j\) denotes the identity map and \(\psi_j\colon I_j(K_{j})\to X\) is an \(L\)-bi-Lipschitz embedding. By construction,
\[
\varphi_\#\big(\bb{\theta}\big)=\sum_{j=1}^\infty (\varphi|_{K_j})_{\#} \bb{\theta|_{K_j}}
\]
and by the same reasoning as in the proof of Proposition~\ref{prop:parametric}, we conclude that
\[
\norm{\, \varphi_{\#}\big(\bb{\theta}\big)\, }_{\bL^\sigma}=\sum_{j=1}^\infty \norm{\, (\varphi|_{K_j})_{\#}\,\big(\bb{\theta|_{K_j}}\big)\, }_{\bL^\sigma}.
\]
Fix \(j\in \N\). To simplify the notation, from now on we omit the index \(j\) throughout. We set \(\varphi=\varphi|_{K_j}\), \(K=K_j\), \(\theta=\bb{\theta|_{K_j}}\),  and also \(I=I_j|_{K_j}\) and \(\psi=\psi_j\).
We have
\[
L^{-n} \cdot \psi_\#\norm{\, I_{\#}\big(\bb{\theta}\big) \, }_{\bL^\sigma} \leq \norm{\, \varphi_{\#}\big(\bb{\theta}\big)\, }_{\bL^\sigma} \leq L^n \cdot \psi_\# \norm{\, I_{\#}\big(\bb{\theta}\big) \, }_{\bL^\sigma},
\]
by the definition of \(\bL^\sigma\). Using that \(\varphi\) is \(L\)-bi-Lipschitz when viewed as a map \((K, \norm{\cdot})\to (X,d)\), we find that 
\[
L^{-1} \cdot \md_p \varphi \leq \md_p I \leq L \cdot \md_p \varphi,
\]
and so \eqref{it:jac-two} and \eqref{it:jac-three} in the definition of Jacobians imply that
\[
L^{-n}\cdot \J^{\sigma}(\md_p \varphi) \leq \J^{\sigma}(\md_p I) \leq L^n \cdot \J^\sigma(\md_p \varphi).
\]
Consequently, taking into account Lemma~\ref{lem:special-case}, we arrive at
\[
L^{-2n} \cdot\Big[\abs{\theta}\, d\mu_{\varphi(K)}^\sigma\Big] \leq \norm{\, \varphi_{\#}\big(\bb{\theta}\big)\, }_{\bL^\sigma} \leq L^{2n} \cdot\Big[\abs{\theta}\, d\mu_{\varphi(K)}^\sigma\Big].
\]
Because of the parametric representation \eqref{eq:parametric-representation}, it follows that \(\norm{T}_{\bL^\sigma}\) and \(\abs{\theta}\, d\mu_S^\sigma\) are \(L^{2n}\)-equivalent. Hence, the statement follows by letting \(L\) tend to \(1\).

\section{Proof of main results}\label{section:six}

We proceed with the proofs of our main results.

\begin{proof}[Proof of Proposition~\ref{prop:equivalent-to-ak}]
It suffices to show that if \(T\) is an \(\bL^\infty\)-current, then \(T\) is a current and \(\norm{T}\leq \norm{T}_{\bL^\infty}\). For any \(\pi\in \Lip(X, \R^n)\), let \(\pi^\ast\in \Lip(X, \R^n)\) be obtained from \(\pi\) by rescaling each coordinate function with \(\big[\Lip(\pi^{(i)})\big]^{-1}\) provided \(\pi^{(i)}\) is not a constant function. Plainly,
\[
\abs{T(f, \pi)}=\Big[\prod_{i=1}^n \Lip(\pi^{(i)}) \Big] \cdot \abs{T(f, \pi^\ast)}
\]
and so since \(\bL^\infty(\pi^\ast)\leq 1\), we find that
\[
\abs{T(f, \pi)}\leq \prod_{i=1}^n \Lip(\pi^{(i)}) \int_X \abs{f} \, d\norm{T}_{\bL^\infty}
\]for all \((f, \pi)\in \mathcal{D}^n(X)\). This shows that \(\norm{T}\leq \norm{T}_{\bL^\infty}\), as desired. 
\end{proof}

The following proof has been communicated to me by an anonymous reviewer.

\begin{proof}[Proof of Proposition~\ref{prop:main}]
There exist constants \(\alpha\), \(\beta>0\) such that 
\[
\alpha \cdot \abs{x}_\infty \leq \sigma(x) \leq \beta \cdot \abs{x}_{\infty}
\]
for all \(x\in \R^n\). Hence, it follows that \(\alpha \cdot \mathbf{L}^\infty\leq \mathbf{L}^\sigma \leq \beta \cdot \mathbf{L}^\infty \). This implies that \(T\) is an \(\bL^\sigma\)-current if and only if it is an \(\bL^\infty\)-current. By Proposition~\ref{prop:equivalent-to-ak} we know that \(T\) is an \(\bL^\infty\)-current if and only if it is a current. Thus, the first part of the proposition follows. 

We now proceed by showing that there exists a constant \(C>0\) only depending on \(n\) such that \eqref{eq:mass-inequality-comparison} holds. Let \(F\colon \R^n \to \R^n\) be a sense preserving linear map that sends the Euclidean unit ball to the John ellipsoid of \(B_\sigma\subset \R^n\). It holds that \(\sigma(F(x))=\abs{x} \cdot \sigma\big(F(\tfrac{x}{\abs{x}})\big) \leq \abs{x}\)
and 
\[
\abs{F^{-1}(y)}=\sigma(y)\cdot \abs{F^{-1}(\tfrac{y}{\sigma(y)})} \leq \sigma(y)\cdot n,
\]
where we used John's ellipsoid theorem for the last inequality. Hence, it follows that
\[
\sigma(F(x))\leq \abs{x} \leq n \cdot \sigma(F(x))
\]
for all \(x\in \R^n\). This implies that \(\bL^\sigma(F\circ \pi) \leq \bL^2(\pi) \leq n \cdot \bL^{\sigma}(F\circ \pi)\) for all \(\pi\in \Lip(X, \R^n)\). Using the chain rule for currents \cite[Theorem 3.5]{ambrosio-2000}, we find that
\[
T(f, \pi)=\frac{1}{\det(F)} \cdot T(f, F\circ \pi)
\]
for all \((f, \pi)\in \mathcal{D}^n(X)\). Hence, it follows from the above that
\[
T(f, \pi) \leq \frac{\big(\bL^\sigma(F\circ \pi)\big)^n}{\det(F)}  \int_X \abs{f}\,\, \text{d}\norm{T}_{\bL^\sigma} \leq \frac{\big(\bL^2(\pi)\big)^n}{\det(F)} \int_X \abs{f}\,\, \text{d}\norm{T}_{\bL^\sigma},
\]
and so \(\norm{T}_{\bL^2}\leq c^{-1}\cdot \norm{T}_{\bL^\sigma}\) for \(c:=\det(F)\). Similarly, 
\[
T(f, F\circ \pi)=\det(F)\cdot T(f, \pi) \leq \det(F)\cdot \big(\bL^2(\pi) \big)^n \int_X \abs{f}\,\, \text{d}\norm{T}_{\bL^2}
\]
and thus
\[
T(f, F\circ \pi) \leq \det(F) \cdot n^n \cdot \big(\bL^{\sigma}(F\circ \pi) \big)^n \int_X \abs{f}\,\, \text{d}\norm{T}_{\bL^2}.
\]
Since \(F\) is invertible, this shows that \(\norm{T}_{\bL^\sigma} \leq c\cdot n^n \cdot \norm{T}_{\bL^2}\). Hence, since \(c=\det(F)=1\), by our assumption on \(\sigma\), the desired inequality \eqref{eq:mass-inequality-comparison} follows. 
\end{proof}

Let us remark here that the proof of Theorem~\ref{thm:main-2} shows that
\[
\norm{T}_{\bL^2} \leq \norm{T}_{\bL^\sigma} \leq n^n \cdot \norm{T}_{\bL^2}.
\]
Hence, the \(\bL^2\)-mass of a current is minimal amongst the \(\bL^\sigma\)-masses. We have the (perhaps naive) impression that this fact should be useful for the study of minimal fillings in metric spaces.

\begin{proof}[Proof of Theorem~\ref{thm:main-2}]
Let \(T\) be an \(n\)-dimensional, rectifiable \(\bL^\sigma\)-current which is concentrated on the \(n\)-rectifiable set \(S\subset X\). Let \((K_i, \phi_i)\) be a bi-Lipschitz parametrization of \(S\). Now, the existence of the desired Borel function \(\theta\colon X \to \R\) follows directly by combining Proposition~\ref{prop:parametric} with Theorem~\ref{thm:formula-for-the-mass-measure}.
\end{proof}

We proceed with the proofs of the results stated in Section~\ref{sec:applics}.

A \textit{Finsler structure} on a smooth connected manifold \(M\) is a continuous map \(\Phi\colon TM \to \R\) such that the restriction to each tangent space is a norm. An assignment \((M, \Phi) \mapsto \vol_\Phi\) which assigns to each \(n\)-dimensional Finsler manifold \((M, \Phi)\) a Borel measure  \(\vol_\Phi\) on \(M\) such that
\begin{enumerate}
    \item \(\vol_\Phi=\Leb^n\) if \(M=\R^n\) and \(\Phi\) the standard Euclidean structure,
    \item \(\vol_\Phi \leq \vol_{\Phi'}\) whenever \(\Phi \leq \Phi'\),
    \item if \(\Phi=\Phi'\circ df\) for some smooth injective map \(f\colon M \to M'\), then \(f_* \vol_{\Phi}=\vol_{\Phi'}|_{f(M)}\),
\end{enumerate}
is called \textit{\(n\)-dimensional Finsler volume functional}. 

Note that each finite-dimensional Banach space \((V, \norm{\cdot})\) naturally carries the Finsler structure \(\Phi(p, v)=\norm{v}\), where we identify \(T_p V\) with \(V\). Hence, any Finsler volume functional induces a Banach volume functional. Conversely, it can be shown (see \cite[Proposition~4.6]{ivanov-2009}) that a Finsler volume functional is uniquely determined by its values on Banach spaces. Thus, from the remarks in Section~\ref{sec:finsler-vols}, it follows that Finsler volume functionals, Banach volume functionals and Jacobians are all in one-to-one correspondences to each other, offering different viewpoints for the same category of objects.

We now briefly recall how Finsler volumes can be described in terms of the measures \(\mu^\bullet\) introduced in Section~\ref{sec:finsler-vols}.
Any Finsler structure \(\Phi\) gives rise to a \textit{Finsler metric} \(d_\Phi\colon M \times M \to \R\) as follows. We set
\[
d_\Phi(x, y) =\inf_{\gamma} \int_{0}^1 \Phi(\gamma^\prime(t))\, dt ,
\]
where the infimum is taken over all piecewise smooth curves \(\gamma\colon [0,1]\to M\) connecting \(x\) to \(y\). Let \(\J^\bullet\) denote a Jacobian and \((M, \Phi)\mapsto \vol_\Phi\) the Finsler volume functional induced by \(\J^\bullet\). Then for every closed Finsler manifold \((M, \Phi)\),
\[
\vol_\Phi=\mu^\bullet_{(M, d_\Phi)}.
\]
This is a straightforward consequence of the definition of \(d_\Phi\). Indeed, if \(\varphi\colon U \to (M, d_\Phi)\) is a bi-Lipschitz map, then for almost every \(p\in U\),  
\[
\md_p \varphi=\norm{D_p \varphi},
\]
where \(\norm{\cdot}=\Phi(x, \cdot)\) for \(x=\varphi(p)\).

\begin{proof}[Proof of Proposition~\ref{prop:explicit-description-L-2}]
Thanks to Theorem~\ref{thm:formula-for-the-mass-measure} it suffices to show that \(T\) admits a parametric representation as in \eqref{eq:parametric-representation} with \(\abs{\theta_i}=1\) for all \(i\in\N\). As \(M\) is compact, the map \(\iota\colon (M, d_g) \to (M, d_\Phi)\) is bi-Lipschitz by construction of the auxiliary Riemannian structure \(g\) on \(M\). Hence, this reduces the problem to show that \(\bb{M}\) admits a good parametric representation. Fix a triangulation \(t\colon \Sigma \to (M, d_g)\) which is bi-Lipschitz on each simplex \(\Delta\subset \Sigma\). Here, \(\Sigma\) is a finite \(n\)-dimensional simplicial complex which is connected and has \(N\) top-dimensional simplices. It follows directly from the Federer-Fleming deformation theorem (see e.g.~\cite[Theorem~A.2]{basso--2023}) that
\begin{equation}\label{eq:rep-of-M}
\bb{M}=\sum_{i=1}^N t_\#\bb{\Delta_i},
\end{equation}
where \(\{\Delta_i\}\) is an enumeration of the \(n\)-simplices of \(\Sigma\) and we use the same convention regarding \(\bb{\Delta}\) as in \cite[p. 20]{basso--2023}. Hence, because of \eqref{eq:rep-of-M}, \(\bb{M}\) admits a representation as in \eqref{eq:parametric-representation} with \(\abs{\theta_i}=1\), as was to be shown. 
\end{proof}

In the following, we use \(B\subset \R^n\) to denote the Euclidean \(n\)-ball and \(\Delta^n \subset \R^n\) the regular \(n\)-simplex with side length one.
Consider a Lipschitz map \(f\colon B \to X\) into a finite-dimensional normed space \(X\). Given an \(n\)-density \(\phi\) on \(X\) the \textit{volume} of \(f\) is defined as follows:
\[
\vol_\phi(f)=\int_B \phi\,\left( (D_p f)(e_1) \wedge \dotsm \wedge (D_p f)(e_n)\right) \, dp,
\]
where \(e_1, \ldots, e_n \in \R^n\) denotes the standard basis. For a Lipschitz chain \(S=\sum \alpha_i \Delta_i\), with \(\alpha_i \in \R\) and \(\Delta_i \colon \Delta^n \to X\), we set \(\vol_\phi(S)=\sum \abs{\alpha_i} \,\vol_\phi (\Delta_i)\). 

Let us now recall the notion of semi-ellipticity of \(\phi\). To give a rigorous definition, we must first clarify what is meant when we say that an embedded \(n\)-ball and a Lipschitz chain in \(X\) have the same boundary. A Lipschitz chain \(S\) is said to be \textit{induced} by the Lipschitz map \(f\colon B \to X\) if the following holds. There exists an oriented finite simplicial complex \(\Sigma\) and a bi-Lipschitz triangulation \(t \colon \Sigma \to B\) of \(B\) such that \(S=f_\#\big([B]_t\big)\) for
\[
[B]_t=\sum_{\Delta \in \mathcal{F}_n }  (t\circ \varphi_\Delta),
\]
where \(\mathcal{F}_n\) denotes the collection of all \(n\)-simplices of \(\Sigma\), and \(\varphi_\Delta \colon \Delta^n \to \Delta\) is the unique isometry that preserves the ordering of the vertices. 

For us, an \textit{\(n\)-dimensional affine ball} is by definition a Lipschitz chain in \(X\) that is induced by a bi-Lipschitz embedding \(f\colon B \to X\) whose image is contained in an \(n\)-dimensional affine subspace of \(X\).

\begin{definition}\label{def:semi-elliptic}
An \(n\)-density \(\phi\) on a normed space \(X\) is called semi-elliptic over \(\R\) if for every \(n\)-dimensional affine ball \(B_{*}\) in \(X\), 
\[
\vol_\phi(B_{*})\leq \vol_\phi(S)
\]
for all Lipschitz chains \(S\) in \(X\) with real coefficients  such that \(\partial S=\partial B_{*}\)
\end{definition}

This notion will be our main tool to prove Proposition~\ref{thm:extendabily-convex-allgemein} below, which includes Theorem~\ref{prop:extendabily-convex} as a special case. In what follows, we suppose that \(\sigma\) is a norm on \(\R^n\) such that \(\J^\sigma(\abs{\, \cdot \, })=1\), that is, \(\J^\sigma\) is a normalized Jacobian. We say that \(\phi\) is induced by \(\J^\sigma\) if the following holds. If \(a=v_1 \wedge \dotsm \wedge v_n\) is a simple \(n\)-vector in \(X\) and \(V=\big\{ x\in X \,\mid\,  a\wedge x=0\big\}\) denotes the 'span' of \(a\), then \(\phi(a)\) is equal to the \(\mu_{V}^{\sigma}\) measure of the parallelotope in \(V\) spanned by the vectors \(v_1, \ldots, v_n\). The measures \(\mu^\bullet\) are defined in Section~\ref{sec:finsler-vols}.

\begin{theorem}\label{thm:extendabily-convex-allgemein}
For any finite-dimensional normed space \(X\), the \(n\)-volume density \(\phi\colon \Lambda^n_s X \to \R\) induced by the Jacobian \(\J^\sigma\) is extendibly convex, that is, \(\phi\) is the restriction of a norm on \(\Lambda^n X\) to the cone of simple \(n\)-vectors.
\end{theorem}

\begin{proof}
In the following, we use the proof strategy of \cite[Theorem~13.2]{ambrosio-2000} adapted to our situation. We abbreviate \(\bL=\bL^{\sigma}\). Any Lipschitz map \(\varphi\colon \Delta^n \to X\) induces an \(\bL\)-current by setting \(\bb{\varphi}=\varphi_\# \bb{\Delta^n}\). In particular, Lipschitz chains with real coefficients induce rectifiable \(\bL\)-currents in the obvious way. Notice that by the definition of the mass measure, 
\[
\left\lVert\, \bb{\sum \alpha_i \Delta_i}\, \right\rVert_{\bL} \leq \sum \, \abs{\alpha_i} \cdot  \big\lVert \bb{\Delta_i}\big\rVert_{\bL}, 
\]
for every Lipschitz chain. By combining Theorem~\ref{thm:formula-for-the-mass-measure} with \cite[Lemma~4]{kirchheim-1994}, we find that
\[
\mass_{\bL}( \bb{\Delta} )= \vol_\phi(\Delta)
\]
for any Lipschitz map \(\Delta \colon \Delta^n \to X\). Therefore, by the above, this yields that 
\begin{equation}\label{eq:ak-trick}
\mass_{\bL}\big(\bb{S}\big) \leq \vol_\phi(S)
\end{equation}
for every Lipschitz chain \(S\) in \(X\) with real coefficients. 

Now, let the Lipschitz chains \(B_{*}\) and \(S\) in \(X\) be given as in Definition~\ref{def:semi-elliptic}.  For all \(\pi \in \Lip(X, \R^n)\), we have that
\begin{equation}\label{eq:relationship-of-phi-area}
\bb{B_{*}}(1, \pi)=\bb{S}(1, \pi) \leq \bL(\pi)^n \cdot \mass_{\bL}(\bb{S}),
\end{equation}
since \(\partial B_{*}=\partial S\) and the inequality follows from the finite mass axiom \ref{it:three}. An application of the area formula shows that there is a closed subset \(A\subset \R^n\) which is bi-Lipschitz equivalent to \(B\) such that \(\bb{B_{*}}=F_\# \bb{A}\) for some injective linear map \(F\colon \R^n \to X\). Let \(V=(\R^n, \norm{\cdot})\) denote the normed space obtained by pulling back the norm on \(F(\R^n)\) via \(F\). Notice that \((F^{-1})_{\#}\bb{B_{*}}=I_{\#}\bb{A}\), and so it follows from Lemma~\ref{lem:special-case} that 
\[
\mass_{\bL}(\bb{B_{*}})=\mass_{\bL}(I_{\#}\bb{A})=\mu^{\sigma}_V(A)=\vol_\phi(B_{*}).
\]
According to Lemma~\ref{lem:desired-inequality}, there exists \(\pi^*\in \Lip(V, \R^n)\) so that \(I_\# \bb{A}(1, \pi^*)=\mass_{\bL}(I_\# \bb{A})\). Since \(\pi=\pi^*\circ F^{-1}\) satisfies \(\bL(\pi)\leq 1\), we find by combining \eqref{eq:relationship-of-phi-area} with \eqref{eq:ak-trick} that  \(\vol_\phi(B_{*})=\bb{B_{*}}(1, \pi) \leq  \vol_\phi(S)\). This shows that \(\phi\) is semi-elliptic over \(\R\). Thus, it follows from \cite[Theorem~3]{ivanov-2004} that \(\phi\) is extendibly convex.
\end{proof}

\subsection{Acknowledgments} I am thankful to Paul Creutz for introducing me to Jacobians and related notions. Moreover, I am indebted to Tommaso Goldhirsch for useful feedback on a first draft version of this article. 

\let\oldbibliography\thebibliography
\renewcommand{\thebibliography}[1]{\oldbibliography{#1}
\setlength{\itemsep}{3pt}}
\bibliographystyle{plain}
\bibliography{sample}

\end{document}